\documentclass[12pt, a4paper]{article}
\usepackage{amssymb,amsmath,amsfonts,amsthm}
\usepackage{setspace}
\usepackage[DIV15]{typearea}
\usepackage[stretch=10]{microtype}
\usepackage{hyperref}

\newcommand{\boA}{\mathcal{A}}

\newcommand{\N}{\mathbb{N}}
\newcommand{\R}{\mathbb{R}}

\newcommand{\norm}[1]{\left\| #1 \right\|}

\DeclareMathOperator{\Var}{Var}

\numberwithin{equation}{section}

\newcommand{\Monm}{\Mon^c}
\DeclareMathOperator{\Mon}{Mon}

\newcommand{\E}{\mathbb{E}}

\newcommand{\p}{ { \mathbb P} }

\newcommand{\F}{\mathcal{F}}

\usepackage{dsfont}
\newcommand{\I}{\mathds{1}}

\theoremstyle{plain}
\newtheorem{thm}{Theorem}[section]
\newtheorem{lma}[thm]{Lemma}
\newtheorem{rmk}[thm]{Remark}
\newtheorem{prop}[thm]{Proposition}

\newtheorem{defn}[thm]{Definition}

\theoremstyle{definition}

\begin{document}

\title{\large \bf The almost sure invariance principle for unbounded functions of
expanding maps}

\author{\Large J. Dedecker\footnote{Universit\'e Paris Descartes, Laboratoire MAP5
and CNRS UMR 8145, 45 rue des Saints P\`eres,
75270 Paris Cedex 06, France.  E-mail: jerome.dedecker@parisdescartes.fr}, S. Gou\"ezel\footnote{Universit\'e Rennes
1, IRMAR and CNRS UMR 6625.
E-mail: sebastien.gouezel@univ-rennes1.fr}, and F.
Merlev\`ede\footnote{Universit\'e Paris Est-Marne la Vall\'ee, LAMA
and CNRS UMR 8050.
Florence.Merlevede@univ-mlv.fr}}

\date{}

\maketitle

\begin{abstract}
We consider two classes of piecewise expanding maps $T$ of $[0,1]$: a
class of uniformly expanding maps for which the Perron-Frobenius
operator has  a spectral gap in the space of bounded variation
functions, and a class of expanding maps with a neutral fixed point
at zero. In both cases, we give a large class of unbounded functions
$f$ for which the partial sums of $f\circ T^i$ satisfy an almost sure
invariance principle. This class contains piecewise monotonic
functions (with a finite number of branches) such that:
\begin{itemize}
\item For uniformly expanding maps, they are square integrable
    with respect to the  absolutely continuous invariant
    probability measure.
\item For maps having a neutral fixed point at zero, they satisfy
    an (optimal) tail condition with respect to the absolutely
    continuous invariant probability measure.
\end{itemize}

\end{abstract}

\section{Introduction and main results}

Our goal in this article is to prove the almost sure invariance
principle with error rate $o(\sqrt{n \ln\ln n})$ for several classes
of one-dimensional dynamical systems, under very weak integrability
or regularity assumptions. We will consider uniformly expanding maps,
and maps with an indifferent fixed point, as defined below.

Several classes of uniformly expanding maps of the interval are
considered in the literature. We will use the very general definition
of Rychlik (1983) to allow infinitely many branches. For notational
simplicity, we will assume that there is a single absolutely
invariant measure and that it is mixing (the general case can be
reduced to this one by looking at subintervals and at an iterate of
the map). We will also need to impose a nontrivial restriction on the
density of the measure: it should be bounded away from $0$ on its
support. This is not always the case, but it is true if there are
only finitely many different images (see Zweim\"uller (1998) for a neat
introduction to such classes of maps, or Broise (1996)).

\begin{defn}\label{UE}
A map $T:[0,1] \to [0,1]$ is uniformly expanding, mixing and with
density bounded from below if it satisfies the following properties:
\begin{enumerate}
\item There is a (finite or countable) partition of $T$ into
    subintervals $I_n$ on which $T$ is strictly monotonic, with a
    $C^2$ extension to its closure $\overline{I_n}$, satisfying
    Adler's condition $|T''|/|T'|^2 \leq C$, and with $|T'|\geq
    \lambda$ (where $C>0$ and $\lambda>1$ do not depend on
    $I_n$).
\item The length of $T(I_n)$ is bounded from below.
\item In this case, $T$ has finitely many absolutely continuous
    invariant measures, and each of them is mixing up to a finite
    cycle. We assume that $T$ has a single absolutely continuous
    invariant probability measure $\nu$, and that it is mixing.
\item Finally, we require that the density $h$ of $\nu$ is
    bounded from below on its support.
\end{enumerate}
\end{defn}
From this point on, we will simply refer to such maps as
\emph{uniformly expanding}. This definition encompasses for instance
piecewise $C^2$ maps with finitely many branches which are all onto,
and with derivative everywhere strictly larger than $1$ in absolute
values.

We consider now a class of expanding maps with a neutral fixed point at zero, 
as defined below.

\begin{defn}
A map $T:[0,1] \to [0,1]$ is a generalized Pomeau-Manneville
map (or GPM map) of parameter $\gamma \in (0,1)$ if there exist
$0=y_0<y_1<\dots<y_d=1$ such that, writing $I_k=(y_k,y_{k+1})$,
\begin{enumerate}
\item The restriction of $T$ to $I_k$ admits a $C^1$ extension
$T_{(k)}$ to $\overline{I_k}$.
\item For $k\geq 1$, $T_{(k)}$ is $C^2$ on $\overline{I_k}$, and $|T_{(k)}'|>1$.
\item $T_{(0)}$ is $C^2$ on $(0, y_1]$, with $T_{(0)}'(x)>1$ for $x\in
(0,y_1]$, $T_{(0)}'(0)=1$ and $T_{(0)}''(x) \sim c
x^{\gamma-1}$ when $x\to 0$, for some $c>0$.
\item $T$ is topologically transitive.
\end{enumerate}
\end{defn}

For such maps, almost sure invariance principles with good remainder
estimates (of the form $O(n^{1/2-\alpha})$ for some $\alpha>0$) have
been established by Melbourne and Nicol (2005) for H\"older
observables, and by Merlev\`ede and Rio (2012) under rather mild
integrability assumptions. For instance, for uniformly expanding
maps, Merlev\`ede and Rio (2012) obtain such a result for a class of
observables $f$ in ${\mathbb L}^p(\nu)$ for $p>2$. This leaves open
the question of the boundary case $f\in {\mathbb L}^2(\nu)$. In this
case, just like in the i.i.d.\ case, one can not hope for a remainder
$O(n^{1/2-\alpha})$ with $\alpha>0$, but one might expect to get
$o(\sqrt{n\ln\ln n})$. This would for instance be sufficient to
deduce the functional law of the iterated logarithm from the corresponding
result for the Brownian motion. The corresponding boundary case for
GPM maps has been studied in Dedecker, Gou\"ezel and Merlev\`ede (2010):
we proved a bounded law of the iterated logarithm (i.e., almost
surely, $\limsup \sum_{i=0}^{n-1} f\circ T^i/\sqrt{n\log\log n}\leq A
<+\infty$), but we were not able to obtain the almost sure invariance
principle.

Our goal in the present article is to solve this issue by combining
the arguments of the two above papers: we will approximate a function
in the boundary case by a function with better integrability
properties, use the almost sure invariance principle of Merlev\`ede and
Rio (2011) for this better function, and show that the bounded law of
the iterated logarithm makes it possible to pass the results from the
better function to the original function. This is an illustration of
a general method in mathematics: to prove results for a wide class of
systems, it is often sufficient to prove results for a smaller (but
dense) class of systems, and to prove uniform (maximal) inequalities.
This strategy gives the almost sure invariance principle in the
boundary case for GPM maps (see Theorem~\ref{ASmap2} below). In the case of uniformly
 expanding maps the almost sure invariance principle for a
dense set of functions has been proved by Hofbauer and Keller (1982)
for a smaller class than that given in Definition \ref{UE}, and
follows from Merlev\`ede and Rio (2012) for the class of uniformly
expanding maps considered in the present paper. However, the bounded law of
the iterated logarithm for the boundary case is not available in the
literature: we will prove it in Proposition~\ref{phi}.

\bigskip

We now turn to the functions for which we can prove the almost sure
invariance principle. The main feature of our arguments is that they
work with the weakest possible integrability condition (merely
${\mathbb L}^2(\nu)$
for uniformly expanding maps), and without any condition on the
modulus of continuity: we only need the functions to be piecewise
monotonic. More precisely, the results are mainly proved for
functions which are monotonic on a single interval, and they are then
extended by linearity to convex combinations of such functions. Such
classes are described in the following definition.

\begin{defn}
If $\mu$ is a probability measure on $\mathbb R$ and $p \in [2,
\infty)$, $M \in (0, \infty)$, let $\Mon_p(M,\mu)$ denote the set of
functions $f:\R\to \R$ which are monotonic on some interval and null
elsewhere and such that $\mu(|f|^p)\leq M^p$. Let $\Monm_p(M,\mu)$ be
the closure in ${\mathbb L}^1(\mu)$ of the set of functions which can
be written as $\sum_{\ell=1}^L a_\ell f_\ell$, where $\sum_{\ell=1}^L
|a_\ell| \leq 1$ and $f_\ell\in \Mon_p(M, \mu)$.
\end{defn}

The above definition deals with ${\mathbb L}^p$-like spaces, with an additional
monotonicity condition. In some cases, it is also important to deal
with spaces similar to weak ${\mathbb L}^p$, where one only requires a uniform
bound on the tails of the functions. Such spaces are described in the
following definition.

\begin{defn}
A function $H$ from ${\mathbb R}_+$ to $[0, 1]$ is a tail function if
it is non-increasing, right continuous, converges to zero at
infinity, and $x\mapsto x H(x)$ is integrable. If $\mu$ is a
probability measure on $\mathbb R$ and $H$ is a tail function, let
$\Mon(H, \mu)$ denote the set of functions $f:\R\to \R$ which are
monotonic on some interval and null elsewhere and such that
$\mu(|f|>t)\leq H(t)$. Let $\Monm(H, \mu)$ be the closure in
${\mathbb L}^1(\mu)$ of the set of functions which can be written as
$\sum_{\ell=1}^L a_\ell f_\ell$, where $\sum_{\ell=1}^L |a_\ell| \leq
1$ and $f_\ell\in \Mon(H, \mu)$.
\end{defn}

\bigskip

Our main theorems follow. For uniformly expanding maps, it involves
an ${\mathbb L}^2$-integrability condition, while for GPM maps the
boundary case is formulated in terms of tails.

\begin{thm} \label{ASmap1} Let $T$ be a uniformly expanding  map
with absolutely continuous invariant measure $\nu$. Then, for any
$M>0$ and any $f \in \Monm_2(M,\nu)$, the series
\begin{equation}\label{var}
\sigma^2=\sigma^2(f)= \nu((f-\nu(f))^2)+ 2 \sum_{k>0}
\nu ((f-\nu(f))f\circ T^k)
\end{equation}
converges absolutely to some nonnegative number.
Moreover,
\begin{enumerate}
\item
On the probability space $([0,1], \nu)$,
the process
\[
   \Big \{\frac{1}{\sqrt n} \sum_{i=0}^{[(n-1)t]} (f \circ T^i
  -\nu(f)), \ t \in [0,1]\Big \}
\]
converges in distribution in the Skorokhod topology to $\sigma W$,
where $W$ is a standard Wiener process.
\item There exists a nonnegative constant $A$ such that
\begin{equation*}
  \sum_{n=1}^\infty \frac{1}{n} \nu \Big( \max_{1 \leq k \leq
  n} \Big |\sum_{i=0}^{k-1} (f \circ T^i
  -\nu(f))\Big|\geq A \sqrt {n \log \log n)} \Big) < \infty \, .
\end{equation*}
\item Enlarging $([0,1], \nu)$ if necessary, there exists a
    sequence $(Z_i)_{i \geq 0}$ of i.i.d.\ Gaussian random
    variables with mean zero and variance $\sigma^2$ defined by
    \eqref{var}, such that
\begin{equation}\label{ASIP}
\Big|\sum_{i=0}^{n-1} (f \circ T^i
  -\nu(f)- Z_i) \Big| =o(\sqrt {n \log \log n})\, ,
  \text{almost surely.}
\end{equation}
\end{enumerate}
\end{thm}

\begin{thm} \label{ASmap2} Let $T$ be a GPM map
with parameter $\gamma \in (0,1/2)$ and invariant measure $\nu$.
Let $H$ be a tail function with
\begin{equation}\label{lilcond}
\int_0^{\infty} x (H(x))^{\frac{1-2\gamma}{1-\gamma}} dx <\infty \,.
\end{equation}
Then, for any $f \in \Monm(H, \nu)$, the series $\sigma^2$ defined in
\eqref{var} converges absolutely to some nonnegative number, and the
asymptotic results 1., 2.\ and 3.\ of Theorem~\ref{ASmap1} hold.
\end{thm}

In particular, it follows from Theorem \ref{ASmap2} that, if $T$ is a
GPM map with parameter $\gamma \in (0,1/2)$, then the almost sure
invariance principle \eqref{ASIP} holds for any positive and
nonincreasing function $f$ on (0,1) such that
\[
  f(x) \leq \frac{C}{x^{(1-2 \gamma)/2}|\ln(x)|^b} \quad \text{near $0$,
  for some $b>1/2$.}
\]
Note that \eqref{ASIP} cannot be true if $f$ is exactly of the form
$f(x)=x^{-(1-2 \gamma)/2}$. Indeed, in that case, Gou\"ezel (2004)
proved that the central limit theorem holds with the normalization
$\sqrt{n \ln(n)}$, and the corresponding almost sure result is
\[
  \lim_{n \rightarrow 0} \frac{1}{\sqrt n (\ln (n))^b} \sum_{i=0}^{n-1} (f \circ T^i
  -\nu(f))= 0 \quad \text{almost everywhere, for any $b>1/2$.}
\]
We refer to the paper by Dedecker, Gou\"ezel and Merlev\`ede (2010) for a deeper
discussion on the optimality of the conditions.

The plan of the paper is as follows. In Section~\ref{secapprox}, we
explain how functions in $\Monm_p(M,\mu)$ or $\Monm(H, \mu)$ can be
approximated by bounded variation functions (to which the results of
Merlev\`ede and Rio (2012) regarding the almost sure invariance
principle apply). In Section~\ref{secasip}, we show how an almost
sure invariance principle for a sequence of approximating processes
implies an almost sure invariance principle for a given process, if
one also has uniform estimates (for instance, a bounded law of the
iterated logarithm). Those two results together with the bounded law
of the iterated logarithm of Dedecker, Gou\"ezel and Merlev\`ede (2010)
readily give the almost sure invariance principle in the boundary
case for GPM maps, as we explain in Section~\ref{secproofasmap2}. In
Section~\ref{secboundedlawphi}, we prove a bounded law of the
iterated logarithm under a polynomial assumption on mixing
coefficients, and we use this estimate in
Section~\ref{secproofasmap1} to obtain the almost sure invariance
principle in the boundary case for uniformly expanding maps,
following the same strategy as above.

\section{Approximation by bounded variation functions} \label{secapprox}
 Let us define the variation $\norm{f}_v$ of a function $f:\R\to \R$
as the supremum of the quantities $|f(a_0)|+\sum_{i=0}^{k-1}
|f(a_{i+1}-f(a_i)|+|f(a_k)|$ over all finite sequences $a_0<\dots
<a_k$. A function $f$ has bounded variation if $\norm{f}_v<\infty$.

In this section, we want to approximate a function in $
\Monm_2(M,\mu)$ or $\Monm(H,\mu)$ in a suitable way. For
$\Monm(H,\mu)$, we shall use the following compactness lemma. It is
mainly classical (compare for instance Hofbauer and Keller (1982)
Lemma 5), but since we have not been able to locate a reference with
this precise statement we will give a complete proof.
\begin{lma}\label{compact}
Let $\mu$ be a probability measure on $\R$. Let $f_n$ be a sequence
of functions on $\R$ with $\norm{f_n}_v \leq C$. Then there exists
$f:\R \to \R$ with $\norm{f}_v\leq C$ such that a subsequence
$f_{\varphi(n)}$ tends to $f$ in $\mathbb{L}^1(\mu)$.
\end{lma}
\begin{proof}
We will first prove that $f_n$ admits a convergent subsequence in
$\mathbb{L}^1(\mu)$. By a classical diagonal argument, it suffices to
show that, for any $\epsilon>0$, one can find a subsequence with
$\limsup_{n \rightarrow \infty} \sup_{m \geq n} \norm{f_{\varphi(n)}- f_{\varphi(m)}}_{\mathbb{L}^1(\mu)}
\leq D\epsilon$, for some $D>0$ not dependending on $\epsilon$.

We consider a finite number of points $a_0<\dots <a_k$ such that
(letting $a_{-1}=-\infty$ and $a_{n+1}=+\infty$), the measure of
every interval $(a_i, a_{i+1})$ is at most $\epsilon$. One can find a
subsequence of $f_n$ such that each $f_{\varphi(n)}(a_i)$ converges,
we claim that it satisfies the desired property. It suffices to show
that a function $g$ with $|g(a_i)|\leq \epsilon$ for all $i$ and
$\norm{g}_v\leq 2C$ satisfies
  \begin{equation}
  \label{pwiu<wcvxlkj}
  \norm{g}_{\mathbb{L}^1(\mu)} \leq D\epsilon.
  \end{equation}
Consider in each interval $(a_i, a_{i+1})$ a point $b_i$ such that
$\sup_{(a_i, a_{i+1})} |g| \leq 2|g(b_i)|$. We have
  \begin{align*}
  \norm{g}_{\mathbb{L}^1(\mu)}
  &
  \leq \sum \mu(a_i, a_{i+1}) \sup_{(a_i, a_{i+1})} |g|
  + \sum \mu\{a_i\} |g(a_i)|
  \\& \leq
  2\sum \mu(a_i, a_{i+1}) (|g(b_i) - g(a_i)| + |g(a_i)|)
  + \sum \mu\{a_i\} |g(a_i)|.
  \end{align*}
Since $|g(a_i)|\leq \epsilon$ and $\mu$ is a probability measure, the
contribution of the terms $|g(a_i)|$ to this expression is at most
$2\epsilon$. Moreover, $\sum \mu(a_i, a_{i+1}) |g(b_i) - g(a_i)| \leq
\epsilon \sum |g(b_i) - g(a_i)| \leq \epsilon\norm{g}_v$. This proves
\eqref{pwiu<wcvxlkj}.

We have proved that $f_n$ admits a subsequence (that we still denote
$f_n$) that converges in $\mathbb{L}^1(\mu)$ to a function $f$.
Extracting further if necessary, we may also assume that it converges
to $f$ on a set $\Omega$ with full measure. On $\overline{\Omega}
-\Omega$, we define $f(x)$ to be $\limsup f(y)$ where $y$ tends to
$x$ in $\Omega$. Finally, on the open set $\R-\overline{\Omega}$
(which may be nonempty if $\mu$ does not have full support), we
define $f(x)$ to be $\max(f(a), f(b))$ where $a$ and $b$ are the
endpoints of the connected component of $x$ in $\R-\overline{\Omega}$
(if one of those endpoints is $-\infty$ or $+\infty$, we only use the
other endpoint). Then $f_n$ converges to $f$ in $\mathbb{L}^1(\mu)$,
and we claim that $f$ has variation at most $C$.

Indeed, consider a sequence $a_0<\dots <a_k$, we want to estimate
$|f(a_0)|+\sum |f(a_{i+1})-f(a_i)|+|f(a_k)|$. Let $b_i=a_i$ if
$a_i\in \Omega$. By construction of $f$, for all $a_i\not\in \Omega$,
one may find a point $b_i$ in $\Omega$ such that $|f(a_i)-f(b_i)|$ is
small, say $<\epsilon/(k+1)$, and we may ensure that $b_0 \leq \dots
\leq b_k$. Then
  \begin{multline*}
  |f(a_0)|+\sum |f(a_{i+1})-f(a_i)|+|f(a_n)|
  \leq 4\epsilon + |f(b_0)|+\sum |f(b_{i+1})-f(b_i)|+|f(b_k)|
  \\
  =4\epsilon + \lim \Bigl(|f_n(b_0)|+\sum |f_n(b_{i+1})-f_n(b_i)|+|f_n(b_k)|\Bigr).
  \end{multline*}
Since the variation of $f_n$ is at most $C$, this is bounded by
$4\epsilon+C$. Letting $\epsilon$ tend to $0$, we get $\norm{f}_v\leq
C$.
\end{proof}

\begin{lma}
\label{lem_approxH}
Let $H$ be a tail function, and consider $f\in \Monm(H,\mu)$. For any
$m>0$, one can write $f=\bar f_m + g_m$ where $\bar f_m$ has bounded variation
and $g_m \in \Monm(H_m, \mu)$ where
$H_m(x)=\min(H(m), H(x))$.
\end{lma}
\begin{proof}
Consider $f\in \Monm(H,\mu)$. By definition, there exists a sequence
of functions $f_L= \sum_{\ell=1}^L a_{\ell, L} g_{\ell, L}$ with
$g_{\ell, L}$ belonging to $\Mon(H, \mu)$ and $\sum_{\ell=1}^L
|a_{\ell, L}| \leq 1$, such that $f_L$ converges in ${\mathbb
L}^1(\mu)$ to $f$. Define then
\[
f_{L,m}=\sum_{\ell=1}^L a_{\ell, L} g_{\ell, L}{\bf 1}_{|g_{\ell, L}| \leq m} \, .
\]
Note that $f_{L,m}$ is such that $\|f_{L,m}\|_v \leq 3m$. Applying Lemma
\ref{compact},
there exists a
subsequence $ f_{\varphi(L), m}$ converging in ${\mathbb L}^1(\mu)$
to a limit $\bar f_m$ such that $\|\bar f_m\|_v \leq 3m$.
Hence $f-\bar f_m$ is the limit in ${\mathbb L}^1(\mu)$ of
\[
f_{\varphi(L)}-f_{\varphi(L), m}=
\sum_{\ell=1}^{\varphi(L)} a_{\ell, \varphi(L)}
g_{\ell, \varphi(L)}{\bf 1}_{|g_{\ell, \varphi(L)}| > m} \, .
\]
Now $g_{\ell, \varphi(L)}{\bf 1}_{|g_{\ell, \varphi(L)}| > m}$
belongs to  $\Mon(\min(H(m), H),\mu)$.
It follows that $f-\bar f_m$ belongs to the class $\Monm(H_m,\mu)$.
\end{proof}

A similar result holds for the space $\Monm_2(M,\mu)$:
\begin{lma}
\label{lem_approxL2}
Consider $f\in \Monm_2(M,\mu)$. For any $m>0$, one can write $f=\bar
f_m + g_m$, where $\bar f_m$ has bounded variation and $g_m \in
\Monm_2(1/m, \mu)$.
\end{lma}

The above proof does not work to obtain this result (the problem is
that the function $g_{\ell} {\bf 1}_{|g_\ell|>m}$ usually does not
satisfy better ${\mathbb L}^2$ bounds than the function $g_\ell$, at
least not uniformly in $g_\ell$). To prove this lemma, we will
therefore need to understand more precisely the structure of elements
of $\Monm_2(M,\mu)$. We will show that they are extended convex
combinations of elements of $\Mon_2(M,\mu)$, i.e., they can be
written as $\int g d\beta(g)$ for some probability measure $\beta$ on
$\Mon_2(M,\mu)$ (the case $\sum a_\ell g_\ell$ corresponds to the
case where $\beta$ is an atomic measure).

To justify this assertion, the first step is to be able to speak of
measures on $\Mon_2(M,\mu)$. We need to specify a topology on
$\Mon_2(M,\mu)$. We use the weak topology (inherited from the space
${\mathbb L}^2(\mu)$, that contains $\Mon_2(M,\mu)$): a sequence $f_n\in
\Mon_2(M,\mu)$ converges to $f$ if, for any continuous compactly
supported function $u:\R\to\R$ (or, equivalently, for any ${\mathbb L}^2(\mu)$
function $u$), $\int f_n(x) u(x) d\mu(x) \to \int f(x) u(x) d\mu(x)$.
\begin{lma}
The space $\Mon_2(M,\mu)$, with the topology of weak convergence, is
a compact metrizable space.
\end{lma}
\begin{proof}
Consider a countable sequence of continuous compactly supported
functions $u_k:\R \to \R$, which is dense in this space for the
topology of uniform convergence. We define a distance on ${\mathbb L}^2(\mu)$
by
  \begin{equation*}
  d(f_1,f_2)= \sum 2^{-k}\min\left(1,\left|\int (f_1-f_2) u_k d\mu \right|\right).
  \end{equation*}
Convergence for this distance is clearly equivalent to weak
convergence.

Let us now prove that $\Mon_2(M,\mu)$ is compact. Consider a sequence
$f_n$ in this space. In particular, it is bounded in ${\mathbb
L}^2(\mu)$. By weak compactness of the unit ball of a Hilbert space,
we can find a subsequence (still denoted by $f_n$) which converges
weakly in ${\mathbb L}^2(\mu)$, to a function $f$. In particular,
$\int f_n u d\mu$ converges to $\int fud\mu$ for any continuous
compactly supported function $u$. Moreover, $f$ is bounded by $M$ in
${\mathbb L}^2(\mu)$. To conclude, it remains to show that $f$ has a
version which is monotonic on an interval, and vanishes elsewhere.

A function in $\Mon_2(M,\mu)$ can be either nonincreasing or
nondecreasing, on an interval which is half-open or half-closed to
the left and to the right, there are therefore eight possible
combinatorial types. Extracting a further subsequence if necessary,
we may assume that all the functions $f_n$ have the same
combinatorial type. For simplicity, we will describe what happens for
one of those types, the other ones are handled similarly. We will
assume that all the functions $f_n$ are nondecreasing on an interval
$(a_n, b_n]$. We may also assume that $a_n$ and $b_n$ are either
constant, or increasing, or decreasing (since any sequence in
$\overline{\R}=\R\cup\{\pm \infty\}$ admits a subsequence with this
property). In particular, those sequences converge in $\overline{\R}$
to limits $a$ and $b$. Let $I$ be the interval with endpoints $a$ and
$b$, where we include $a$ in $I$ if $a_n$ is increasing and exclude
it otherwise, and where we include $b$ if $b_n$ is decreasing or
constant and exclude it otherwise. The Banach-Saks theorem shows that
(extracting further if necessary) we may ensure that the sequence of
functions $g_N=\frac{1}{N}\sum_{n=1}^N f_n$ converges to $f$ in
$\mathbb{L}^2(\mu)$ and on a set $A$ of full measure. It readily
follows that $f$ is nondecreasing on $A\cap I$ and vanishes on $A\cap
(\R-I)$. Modifying $f$ on the zero measure set $\R-A$, we get a
function in $\Mon_2(M,\mu)$ as claimed.
\end{proof}
The Borel structure coming from the weak topology on ${\mathbb L}^2(\mu)$ coincides
with the Borel structure coming from the norm topology (since an open
ball for the norm topology can be written as a countable intersection
of open sets for the weak topology, by the Hahn-Banach theorem).
Therefore, all the usual functions on $\Mon_2(M,\mu)$ are measurable.

If $\beta$ is a probability measure on $\Mon_2(M,\mu)$, we can define
a function $f\in {\mathbb L}^2(\mu)$ by $f(x)=\int g(x) d\beta(g)$. We claim
that the elements of $\Monm_2(M,\mu)$ are exactly such functions:
\begin{prop}
\label{prop_structure}
We have
  \begin{equation*}
  \Monm_2(M,\mu) = \left\{ \int_{\Mon_2(M,\mu)} g d\beta(g)\, :\,
  \beta\text{ probability measure on }\Mon_2(M,\mu)\right\}.
  \end{equation*}
\end{prop}
\begin{proof}
We have two inclusions to prove.

Consider first $f\in \Monm_2(M,\mu)$, we will show that it can be
written as $\int g d\beta(g)$ for some measure $\beta$. By definition
of $\Monm_2(M,\mu)$, there exists a sequence of atomic probability
measures $\beta_n$ on $\Mon_2(M,\mu)$ such that $f_n=\int g
d\beta_n(g)$ converges in ${\mathbb L}^1(\mu)$ to $f$. Since the
space $\Mon_2(M,\mu)$ is compact, the sequence of measures $\beta_n$
admits a convergent subsequence (that we still denote by $\beta_n$),
to a measure $\beta$. By definition of vague convergence, for any
continuous function $\Psi$ on $\Mon_2(M,\mu)$, $\int \Psi(g)
d\beta_n(g)$ tends to $\int \Psi(g) d\beta(g)$. Fix a continuous
compactly supported function $u$ on $\R$. By definition of the
topology on $\Mon_2(M,\mu)$, the map $\Psi_u:g\mapsto \int
u(x)g(x)d\mu(x)$ is continuous. Therefore, $\int
\Psi_u(g)d\beta_n(g)$ tends to $\int \Psi_u(g)d\beta(g)$, i.e., $\int
u(x)f_n(x)d\mu(x)$ tends to $\int u(x)f_\beta(x)d\mu(x)$, where
$f_\beta=\int g d\beta(g)$. This shows that $f_n$ converges weakly to
$f_\beta$. However, by assumption, $f_n$ converges in ${\mathbb
L}^1(\mu)$ to $f$. We deduce that $f=f_\beta$, as desired.

Conversely, consider a function $f_\beta$ for some probability
measure $\beta$ on $\Mon_2(M,\mu)$, let us show that it belongs to
$\Monm_2(M,\mu)$. Let us consider a sequence of atomic probability
measures $\beta_n$ converging vaguely to $\beta$. The arguments in
the previous paragraph show that the functions $f_{\beta_n}$ converge
weakly to $f_\beta$. By Banach-Saks theorem, extracting a subsequence
if necessary, we can ensure that $f_N=N^{-1}\sum_{n=1}^N f_{\beta_n}$
converges almost everywhere and in $\mathbb{L}^2(\mu)$ to $f_\beta$.
In particular, it converges to $f_\beta$ in ${\mathbb L}^1(\mu)$.
Since $f_N$ can be written as $\sum a_{\ell,N}f_{\ell,N}$ for some
functions $f_{\ell,N}\in \Mon_2(M,\mu)$ and some coefficients
$a_{\ell,N}$ with sum bounded by $1$, this shows that $f_\beta$
belongs to $\Monm_2(M,\mu)$.
\end{proof}

\begin{proof}[Proof of Lemma~\ref{lem_approxL2}]
Consider $f\in \Monm_2(M,\mu)$, and $\epsilon>0$. By
Proposition~\ref{prop_structure}, there exists a measure $\beta$ on
$\Mon_2(M,\mu)$ such that $f=\int g d\beta(g)$. For each $g\in
\Mon_2(M,\mu)$, let $K(g)$ be the smallest number such that $\int g^2
1_{|g|\geq K(g)} \leq \epsilon^2$. Fix some $K>0$. We have
  \begin{align*}
  f(x)&=\int_{K(g)< K} g(x) d\beta(g) + \int_{K(g)\geq K} g(x) d\beta(g)
  \\&= \int_{K(g)<K} g(x) {\bf 1}_{|g(x)|\leq K(g)} d\beta(g)
  + \int_{K(g)<K} g(x) {\bf 1}_{|g(x)|> K(g)} d\beta(g)
  + \int_{K(g)\geq K} g(x) d\beta(g).
  \end{align*}
The first term has variation bounded by $3K$. In the second term,
each function $g {\bf 1}_{|g|> K(g)}$ is monotonic on an interval and
null elsewhere, with ${\mathbb L}^2(\mu)$ norm bounded by $\epsilon$.
Therefore, the second term belongs to $\Monm_2(\epsilon, \mu)$.
Writing $A(K)=\{g\,:\, K(g)\geq K\}$ and $\alpha(K)=\beta(A(K))$, the
third term is the average over $A(K)$ of the functions $\alpha(K)
g\in \Mon_2(\alpha(K)M,\mu)$ with respect to the probability measure
${\bf 1}_{A(K)} d\beta(g)/\alpha(K)$. Therefore, it belongs to
$\Monm_2(\alpha(K)M, \mu)$. Taking $K$ large enough so that
$\alpha(K)M \leq \epsilon$, we infer that $f$ is the sum of a
function of bounded variation and a function in $\Monm_2(2\epsilon,
\mu)$.
\end{proof}

\section{Strong invariance principle by approximation}\label{secasip}

Let $(X_i)_{i \geq 1}$ be a sequence of random variables. Assume that
\begin{enumerate}
\item For each $m \in \N$ there exists a sequence $(X_{i,m})_{i
    \geq 1}$ such that
  \begin{equation*}
  \limsup_{n\to\infty} \left| \frac{ \sum_{i=1}^n X_i - X_{i,m}}{\sqrt{n\log \log n}}\right|
  \leq \epsilon(m)  \quad \text{almost surely},
  \end{equation*}
where $\epsilon(m)$ tends to $0$ as $m$ tends to infinity.
\item For each $m \in \N$, the sequence  $(X_{i,m})_{i \geq 1}$
    satisfies a strong invariance principle: there exists a
    sequence $(Z_{i,m})_{i \geq 1}$ of i.i.d.\ Gaussian random
    variables with mean $0$ and variance $\sigma^2_m$ such that
  \begin{equation*}
 \lim_{n \rightarrow \infty}
  \frac{ \sum_{i=1}^n X_{i,m}-Z_{i,m}}{\sqrt{n\log \log n}}=0 \quad \text{almost surely.}
  \end{equation*}
We also assume that $\sigma_m^2$ converges as  $m\to\infty$ to a
limit $\sigma^2$.
\item  There exists an infinite subset $\boA$ of $\N$ such that,
    for any $A\in \boA$, the
$\sigma$-algebras $\sigma(Z_{i,m})_{i<A,\, m\in \N}$ and
$\sigma(Z_{i,m})_{i\geq A,\, m\in \N}$ are independent.
\end{enumerate}

\begin{prop}\label{diag}
Under the assumptions 1, 2 and 3, there exists a sequence  $(Z_i)_{i
\geq 1}$ of i.i.d.\ Gaussian random variables with mean zero and
variance $\sigma^2$ such that
  \begin{equation}\label{asipfinal}
  \lim_{n\to\infty} \frac{ \sum_{i=1}^n X_i-Z_i}{\sqrt{n\log \log n}} = 0 \quad
  \text{almost surely.}
  \end{equation}
\end{prop}
\begin{proof}
The idea of the proof is to use a diagonal argument: we will use the
$Z_{i,0}$ for some time, then the $Z_{i,1}$ for a longer time, and so
on, to construct the $Z_i$.

Let $A_m$ be a sequence of elements of $\boA$ tending to infinity
fast enough. More precisely, we choose $A_m$  in such a way that
there exists a set $\Omega_m$ with probability greater than
$1-2^{-m}$ on which, for any $n\geq A_m$,
\begin{equation*}
  \left|\frac{ \sum_{i=1}^n X_{i,m}-Z_{i,m}}{\sqrt{n\log \log n}}\right| \leq \epsilon(m)
\quad \text{and} \quad
  \left| \frac{ \sum_{i=1}^n X_i - X_{i,m}}{\sqrt{n\log \log n}}\right|
  \leq 2\epsilon(m).
\end{equation*}
The assumptions 1 and 2 ensure that these two properties are
satisfied   provided $A_m$ is large enough. We also choose $A_m$ in a
such a way that, for  $j<m-1$,
\begin{equation}
  \label{decay}
   \epsilon(j)\sqrt{A_{j+1} \log\log A_{j+1}} < 2^{-(m-j)} \epsilon(m)\sqrt{A_m \log \log A_m}.
  \end{equation}
Indeed, if the  $A_j$'s have been defined for $j<m$, it suffices to
take $A_m$ large enough for \eqref{decay} to hold.

With this choice of $A_m$, we infer that for any $\omega\in \Omega_m$
and any $n\geq A_m$,
  \begin{equation*}
  \left|\sum_{i=1}^n X_i-Z_{i,m}\right|\leq 3\epsilon(m) \sqrt{n\log\log n}.
  \end{equation*}
  Hence, for any $\omega\in
\Omega_m$ and any $n\geq A_m$,
 \begin{equation}
 \label{intermed}
  \left|\sum_{i=A_m}^n X_i-Z_{i,m}\right|\leq 6\epsilon(m) \sqrt{n\log\log n}.
 \end{equation}

For $i\in [A_m, A_{m+1}-1]$, let $m(i)=m$. Let $(\delta_k)_{k \geq
1}$ be a sequence of i.i.d.\ Gaussian random variables with mean zero
and variance $\sigma^2$, independent of the array $(Z_{i,m})_{i\geq
1, m \geq 1}$. We now construct the sequence $Z_i$ as follows: if
$\sigma_{m(i)}=0$, then $Z_i=\delta_i$, else $Z_i=
(\sigma/\sigma_{m(i)}) Z_{i, m(i)}$. By construction, thanks to the
assumption 3, the $Z_i$'s are i.i.d.\ Gaussian random variables with
mean zero and variance $\sigma^2$. Let us show that they satisfy
\eqref{asipfinal}.

Let $D_i=Z_i-Z_{i,m(i)}$ and note that $(D_i)_{i \geq 1}$ is a
sequence of independent Gaussian random variables with mean zero and
variances $\Var(D_i)=(\sigma-\sigma_{m(i)})^2$. Since $\sigma_{
m(i)}$ converges to $\sigma$ as $i$ tends to infinity, it follows
that
\[
  \text{letting} \quad v_n= \frac 1 n \Var\Big( \sum_{i=1}^n D_i \Big), \quad \text{then}
  \quad \lim_{n \rightarrow  \infty} v_n=0.
\]
From the basic inequality
\[
  {\mathbb P}\Big
  ( \max_{1 \leq k \leq n} \Big| \sum_{i=1}^k D_i \Big|>x\Big) \leq
  2 \exp \Big(-\frac{ x^2}{2 n v_n}\Big)\, ,
\]
it follows that
  \begin{equation*}
  \lim_{n \rightarrow \infty}
  \frac{\sum_{i=1}^n Z_{i,m(i)}-Z_i}{\sqrt{n\log\log n}} = 0 \quad \text{almost surely.}
  \end{equation*}
To conclude the proof, it remains to prove that
  \begin{equation}
  \label{eq_to_prove}
  \lim_{n \rightarrow \infty}
   \frac{\sum_{i=1}^n X_i - Z_{i,m(i)}}{\sqrt{n \log\log n}} = 0 \quad
   \text{almost surely}.
  \end{equation}
Let $B=\{ \omega : \omega \in \liminf \Omega_m \}$. By
 Borel-Cantelli, ${\mathbb P}(B)=1$. For $\omega \in B$,  there
 exists $m_0(\omega)$ such that  $\omega$ belongs to all the
 $\Omega_m$ for $m\geq
m_0(\omega)$. For $n\geq A_{m_0(\omega)}$, we have  (denoting by $M$
the greater integer such that $A_M \leq n$)
  \begin{equation*}
  \left|\sum_{i=1}^n X_i - Z_{i,m(i)}\right|
  \leq \sum_{i=1}^{A_{m_0(\omega)}-1}|X_i-Z_{i,m(i)}|
  + \sum_{m=m_0(\omega)}^{M-1} \left|\sum_{i=A_m}^{A_{m+1}-1}X_i - Z_{i,m}\right|
  + \left| \sum_{i=A_M}^n X_i-Z_{i,M}\right|.
  \end{equation*}
Taking into account   \eqref{decay} and \eqref{intermed}, we obtain
  \begin{align*}
   \left|\sum_{i=1}^n X_i - Z_{i,m(i)}\right|
  &
   \leq C(\omega)
   + \sum_{m=1}^{M-1} 6\epsilon(m) \sqrt{A_{m+1} \log\log A_{m+1}}
   + 6\epsilon(M) \sqrt{n\log\log n}
  \\&
   \leq C(\omega) + \sum_{m=1}^{M-2} 6\epsilon(M) \sqrt{A_M \log\log A_M} 2^{-(M-m)}
   \\&\ \ \ %
   + 6\epsilon(M-1) \sqrt{A_M \log\log A_M}
   +6\epsilon(M) \sqrt{n\log\log n}
  \\&
   \leq C(\omega) + 9(\epsilon(M-1)+\epsilon(M)) \sqrt{n\log \log n}.
  \end{align*}
Since $\epsilon(M-1)+\epsilon(M)$ tends to  zero as $n$ tends to
infinity, this proves \eqref{eq_to_prove} and completes the proof of
Proposition~\ref{diag}.
\end{proof}

\begin{rmk}
The proposition would also apply to random variables taking values in
$\R^d$ or in Banach spaces (with the same proof), but we have
formulated it only for real-valued random variables in view of our
applications. Indeed, the class of functions we consider relies on
monotonicity which is a purely one-dimensional notion.
\end{rmk}

\section{Proof of Theorem~\ref{ASmap2} on GPM maps}
\label{secproofasmap2}
To prove Theorem~\ref{ASmap2}, we should establish the convergence of
the series \eqref{var} as well as the asymptotic results 1., 2.\ and
3.\ described in Theorem~\ref{ASmap1}. The convergence of \eqref{var}
and the asymptotics 1.\ and 2.\ have been proved in Dedecker, Gou\"ezel
and Merlev\`ede (2010). Therefore it only remains to prove the almost
sure invariance principle.

To do this, we apply Proposition~\ref{diag} to the sequences $X_i=
f\circ T^i -\nu(f)$ and $X_{i,m}=\bar f_m \circ T^i - \nu(\bar f_m)$,
where the function $\bar f_m$ has been constructed in
Lemma~\ref{lem_approxH}. Let us denote by $S_n (f )= \sum_{i=0}^{n-1} (f \circ T^i - \nu (f))$. To apply Proposition~\ref{diag}, we have to
check the assumptions 1., 2.\ and 3.\ of Section~\ref{secasip}.

The function $g_m=f-\bar f_m$ belongs to $\Monm(H_m,\nu)$ where
$H_m=\min (H(m),H)$, by Lemma~\ref{lem_approxH}. Therefore, it
belongs to the class of functions to which the results of Dedecker,
Gou\"ezel and Merlev\`ede (2010) apply: $S_n ( g_m)$ satisfies a central limit
theorem and a bounded law of the iterated logarithm. In particular,
applying Theorem 1.5 of this article (and Section 4.5 there to
compute the constant $M(m)$) we get that, almost surely,
  \begin{equation*}
  \limsup \frac{1}{\sqrt{n\log \log n}} \left|\sum_{i=0}^{n-1} (g_m \circ T^i-\nu(g_m)) \right| \leq M(m),
  \end{equation*}
where $M(m)= C \int_0^{\infty} x
(H_m(x))^{\frac{1-2\gamma}{1-\gamma}} dx$, $C$ being some positive
constant. Since $M(m)$ tends to zero as $m$ tends to infinity, the
assumption 1.\ of Section~\ref{secasip} follows by choosing
$\epsilon(m)=2M(m)$.

Since the function $\bar f_m$ has bounded variation
%belongs to $\Monm({\bf 1}_{[0,
%m]},\nu)$,
we can apply Item 2 of Theorem 3.1 of Merlev\`ede and Rio
(2012) to the sequence $(X_{i,m})$ (see their Remark 3.1 for the case
of GPM maps). Hence there exists a sequence $(Z_{i,m})_{i \geq 1}$ of
i.i.d.\ Gaussian random variables with mean $0$ and variance
$\sigma^2_m=\sigma^2(\bar f_m)$ such that
\begin{equation*}
 \lim_{n \rightarrow \infty}
  \frac{ \sum_{i=1}^n X_{i,m}-Z_{i,m}}{\sqrt{n\log \log n}}=0 \quad \text{almost surely.}
\end{equation*}
More precisely, it follows from their construction (see the
definition of the variables $V^*_{k,L}$ in Section 4.2 of Merlev\`ede
and Rio (2010)) that the assumption 3.\ of Section~\ref{secasip} is
satisfied with ${\mathcal A}=\{ 2^L, L \in {\mathbb N}^*\}$.

To check the assumption 2.\ of Section~\ref{secasip}, it remains only
to prove that $\sigma^2_m$ converges to $\sigma^2$ as $m$ tends to
infinity. We have $f=\bar f_m+g_m$, therefore
  \begin{equation*}
  \frac{S_n (f )}{\sqrt{n}}=\frac{S_n (\bar f_m)}{\sqrt{n}}
  + \frac{S_n ( g_m)}{\sqrt{n}}\, .
  \end{equation*}
The term on the left converges in distribution to a Gaussian with
variance $\sigma^2$, and the terms on the right converge to

(non-independent) Gaussians with respective variances $\sigma_m^2$
and $\sigma^2(g_m)$. To conclude, it suffices to show that
$\sigma^2(g_m)$ converges to $0$ when $m$ tends to infinity.

As we have explained above, the results of Dedecker,
Gou\"ezel and Merlev\`ede (2010) apply, and show that $S_n(g_m)$ satisfies
a central limit theorem. From the same paper (see Sections 2.2 and 4.1 there), 
we get  the following estimate
on the asymptotic variance $\sigma^2(g_m)$ of $n^{-1/2} S_n(g_m)$~: there
exists a positive constant $C$ such that
\[
\sigma^2(g_m) \leq C\int_0^{\infty} x (H_m(x))^{\frac{1-2\gamma}{1-\gamma}} dx \, ,
\]
and the second term on right hand tends to zero as $m$ tends to
infinity by using \eqref{lilcond} and the dominated convergence
theorem. The result follows.

Hence, we have checked that the assumptions 1., 2.\ and 3.\ of
Section~\ref{secasip} are satisfied. This completes the proof of  the
almost sure invariance principle. \qed

\section{A bounded LIL for
\texorpdfstring{$\phi$}{phi}-dependent sequences}
\label{secboundedlawphi}

Let $(\Omega ,\mathcal{A}, \p)$ be a probability space, and let
$\theta :\Omega \mapsto \Omega $ be a bijective bimeasurable
transformation preserving the probability ${\p}$. Let ${\mathcal
F}_0$ be a sub-$\sigma$-algebra of $\mathcal{A}$ satisfying
${\mathcal F}_0 \subseteq \theta^{-1}({\mathcal F}_0)$.

\begin{defn}\label{defphi}
For any integrable random variable $X$, let us write
$X^{(0)}=X- \E(X)$.
For any random variable $Y=(Y_1, \cdots, Y_k)$ with values in
${\mathbb R}^k$ and any $\sigma$-algebra $\F$, let
\[
\phi(\F, Y)= \sup_{(x_1, \ldots , x_k) \in {\mathbb R}^k}
\left \| \E \Big(\prod_{j=1}^k (\I_{Y_j \leq x_j})^{(0)} \Big | \F \Big)^{(0)}
\right\|_\infty.
\]
For a sequence ${\bf Y}=(Y_i)_{i \in {\mathbb Z}}$, where $Y_i=Y_0
\circ \theta^i$ and $Y_0$ is an $\F_0$-measurable and real-valued
random variable, let
\begin{equation*}
\phi_{k, {\bf Y}}(n) = \max_{1 \leq l \leq
k} \ \sup_{ n\leq i_1\leq \ldots \leq i_l} \phi(\F_0,
(Y_{i_1}, \ldots, Y_{i_l})) .
\end{equation*}
\end{defn}

The interest of those mixing coefficients is that they are not too
restrictive, so they can be used to study several classes of dynamical
systems, and that on the other hand they are strong enough to yield
correlation bounds for piecewise monotonic functions (or, more
generally, functions in $\Monm_p(M,\mu)$). In particular, we have the
following:
\begin{lma}\label{covaphi}
Let ${\bf Y}=(Y_i)_{i \in {\mathbb Z}}$, where $Y_i=Y_0 \circ
\theta^i$ and $Y_0$ is an $\F_0$-measurable random variable. Let $f$
and $g$ be two functions from ${\mathbb R}$ to ${\mathbb R}$ which
are monotonic on some interval and null elsewhere. Let $p \in [1,
\infty]$. If $\|f(Y_0)\|_p < \infty$, then, for any positive integer
$k$,
\[
  \|{\mathbb E}(f(Y_k)|{\mathcal F}_0)-{\mathbb E}(f(Y_k))\|_p \leq 2
  (2\phi_{1, {\bf Y}}(k))^{(p-1)/p}\|f(Y_0)\|_p \, .
\]
If moreover $p\geq 2$ and $\|g(Y_0)\|_p < \infty$, then for any positive integers
 $i\geq j\geq k$,
\[
\|{\mathbb E}(f(Y_i)^{(0)}g(Y_j)^{(0)}|{\mathcal F}_0)-
{\mathbb E}(f(Y_i)^{(0)}g(Y_j)^{(0)})\|_{p/2}\leq 8
 (4\phi_{2, {\bf Y}}(k))^{(p-2)/p}\|f(Y_0)\|_{p}\|g(Y_0)\|_p \, .
\]
\end{lma}
\begin{proof}
Note first that, for any positive integers $i\geq j \geq k$,
\begin{align*}
  \phi({\mathcal F}_0, f(Y_k)) &\leq  2 \phi ({\mathcal F}_0, Y_k)
  \leq 2\phi_{1, {\bf Y}}(k),\\
  \phi({\mathcal F}_0, (f(Y_j), g(Y_i))) &\leq   4
  \phi ({\mathcal F}_0, (Y_j, Y_i))\leq 4 \phi_{2, {\bf Y}}(k)\, .
\end{align*}
This follows from definition \eqref{defphi}, by noting that $\{f\leq
t\}$ (and also $\{ g \leq s\}$) is either an interval or the
complement of an interval.

To prove the first inequality of the lemma, let us note that
\[
\|{\mathbb E}(f(Y_k)|{\mathcal F}_0)-{\mathbb E}(f(Y_k))\|_p =
\sup_{Z \in B_{p/(p-1)} ({\mathcal F}_0)}{\mathrm{Cov}}(Z, f(Y_k)) \, ,
\]
where $B_{q} ({\mathcal F}_0)$ is the set of ${\mathcal
F}_0$-measurable random variables $Z$ such that $\|Z\|_q \leq 1$.
Proposition 2.1 of Dedecker (2004) states that $| \mathrm{Cov}(Z, Y) |
\leq 2 \phi(\sigma(Z),Y)^{(p-1)/p} \|Y\|_p \|Z\|_{p/(p-1)}$. Since $\phi(\sigma(Z),
f(Y_k)) \leq \phi({\mathcal F}_0, f(Y_k)) \leq 2\phi_{1, {\bf
Y}}(k)$, we obtain the first inequality of Lemma~\ref{covaphi} as
desired.

For the second inequality, we note in the same way that
\[
\|{\mathbb E}(f(Y_i)^{(0)}g(Y_j)^{(0)}|{\mathcal F}_0)-
{\mathbb E}(f(Y_i)^{(0)}g(Y_j)^{(0)})\|_{p/2}
=\sup_{Z \in B_{p/(p-2)} ({\mathcal F}_0)}{\mathrm{Cov}}(Z, f(Y_i)^{(0)} g(Y_j)^{(0)})\, .
\]
Proposition 6.1 of Dedecker, Merlev\`ede and Rio (2009) gives a control
of the covariance in terms of $\phi({\mathcal F}_0, (f(Y_j),
g(Y_i)))$. Since this quantity is bounded by $4 \phi_{2, {\bf
Y}}(k)$, the result follows.
\end{proof}

The main result of this section is the following proposition, showing
that a suitable polynomial assumption on mixing coefficients implies
a bounded law of the iterated logarithm for piecewise monotonic ${\mathbb L}^2$
functions.

\begin{prop}\label{phi}
Let  $X_i = f(Y_i) - \E ( f(Y_i))$, where $Y_i=Y_0 \circ \theta^i$
and $Y_0$ is an $\F_0$-measurable random variable. Let
\[S_n=S_n(f)=\sum_{k=1}^n X_k \, ,\]
and let $P_{Y_0}$ be the distribution of $Y_0$. Assume that
\begin{equation}\label{DDM}
\sum_{k \geq 1} k^{1/\sqrt 3 -1/2} \phi^{1/2}_{2, {\bf Y}}(k) < \infty\, .
\end{equation}
If $f$ belongs to $\Monm_2(M,P_{Y_0})$ for some $M>0$, then
\begin{equation} \label{but2}
\sum_{n >0}\frac{1}{n} \p \Big( \max_{1 \leq k \leq n} |S_k| > 3C M
\sqrt{  n \log \log n}\Big) < \infty \, ,
\end{equation}
where $ C = 16 \sum_{ k \geq 0} \phi_{1, {\bf Y}}^{1/2} (k)$.
\end{prop}
\begin{proof}
Let $f \in \Monm_2(M,P_{Y_0})$. By definition of
$\Monm_2(M,P_{Y_0})$, there exists $f_L= \sum_{\ell=1}^L a_{\ell, L}
g_{\ell, L}$ with $g_{\ell, L}$ belonging to $\Mon_2(M, P_{Y_0})$ and
$\sum_{\ell=1}^L |a_{\ell, L}| \leq 1$, and such that $f_L$ converges
in ${\mathbb L}^1(P_{Y_0})$ to $f$. It follows that $X_{i,L}=f_L(Y_i)
- \E ( f_L(Y_i))$ converges in ${\mathbb L}^1$ to $X_i$ as $L$ tends
to infinity. Extracting a subsequence if necessary, one may also
assume that the convergence holds almost surely.

Hence, for any fixed $n$, $S_{n}(f_L)=\sum_{k=1}^n X_{k,L}$ converges almost surely
and in ${\mathbb L}^1$ to $S_n(f)$. Assume that one can prove that,
for any positive integer $L$,
\begin{equation} \label{but2bis}
\sum_{n >0}\frac{1}{n} \p \Big( \max_{1 \leq k \leq n} |S_{k}(f_L)| > 3CM
\sqrt{ n \log \log n}\Big) < K \, ,
\end{equation}
for some positive constant $K$ not depending on $L$. Let us explain
why \eqref{but2bis} implies \eqref{but2}. Let $Z_n= \max_{1 \leq k
\leq n} |S_{k}(f)|/\sqrt{ M^2 n \log \log n}$. By Beppo-Levi,
\begin{equation}\label{BeppoLevy}
\sum_{n >0}\frac{1}{n} \p \Big( \max_{1 \leq k \leq n} |S_k(f)| > 3C
M\sqrt{ n \log \log n}\Big) =
\lim_{k \rightarrow \infty} \E \Big (
\sum_{n >0}\frac{1}{n} {\bf 1}_{ Z_n > 3C+k^{-1}}
\Big).
\end{equation}
Let $h_k$ be a continuous function from ${\mathbb R}$ to $[0, 1]$, such that
$h_k(x)=1$ if $x >3C + k^{-1}$ and $h_k(x)=0$ if $x < 3C$. Let
$Z_{n,L}= \max_{1 \leq k \leq n} |S_{k, L}|/\sqrt{ M^2 n \log \log n}$.
By Fatou's lemma,
\begin{multline}\label{Fatou}
\E \Big (
\sum_{n >0}\frac{1}{n} {\bf 1}_{ Z_n > 3C+k^{-1}}
\Big)\leq \E \Big (
\sum_{n >0}\frac{1}{n} h_k(Z_n)
\Big) \\
\leq \liminf_{L \rightarrow \infty}
\E \Big (
\sum_{n >0}\frac{1}{n} h_k(Z_{n, L})
\Big)\leq \liminf_{L \rightarrow \infty}
\E \Big (
\sum_{n >0}\frac{1}{n} {\bf 1}_{Z_{n, L}>3C}
\Big)\, .
\end{multline}
From \eqref{but2bis}, \eqref{BeppoLevy} and \eqref{Fatou}, we infer
that
\[
\sum_{n >0}\frac{1}{n} \p \Bigl( \max_{1 \leq k \leq n} |S_k(f)| > 3C
\sqrt{M(f)n \log \log n}\Bigr)
\leq  \liminf_{L \rightarrow \infty}
\E \Bigl (\sum_{n >0}\frac{1}{n} {\bf 1}_{Z_{n, L}>3C} \Bigr)
\leq  K \, ,
\]
and \eqref{but2} follows.

Hence, it remains to prove \eqref{but2bis}, or more generally that:
if $f= \sum_{\ell=1}^L a_{\ell} f_{\ell}$ with $f_{\ell}$ belonging
to $\Mon_2(M, P_{Y_0})$ and $\sum_{\ell=1}^L |a_{\ell}| \leq 1$, then
\begin{equation} \label{but2ter}
\sum_{n >0}\frac{1}{n} \p \Big( \max_{1 \leq k \leq n} |S_{k}(f)| > 3C M
\sqrt{ n \log \log n}\Big) < K \, ,
\end{equation}
for some positive constant $K$ not depending on $f$.

We now prove \eqref{but2ter}. We will need to truncate the functions.
It turns out that the optimal truncation level is at
$\sqrt{n}/\sqrt{\log \log n}$: the large part can then be controlled
by a simple $\mathbb{L}^1$ estimate, while the truncated part can be
estimated thanks to a maximal inequality of Pinelis (1994) (after a
reduction to a martingale). Let
%$b = \sqrt {8 } M$, and let
$ g_n(x)
= x {\bf 1}_{|x| \leq M n^{1/2}/\sqrt{ \log \log n}} $.
%(x \wedge b n^{1/2}/\sqrt{ \log \log n}) \vee (-b n^{1/2}/\sqrt{ \log \log n})\, .
For
any $i \geq 0$, we first define
\[
 X_{i,n}'=\sum_{\ell=1}^L a_{\ell} \, g_n \circ
 f_{\ell}(Y_i) -\sum_{\ell =1}^L a_{\ell}
  \E ( g_n \circ f_{\ell}(Y_i)) \quad \text{and} \quad  X_{i,n}''=X_i - X_{i,n}' \, .
\]

Let
\[ d_{i,n} = \sum_{ j \geq i} \E(X'_{j,n}|\mathcal{F}_i)- \E(X'_{j,n}|\mathcal{F}_{i-1}) \quad \text{and}
\quad M_{k,n}= \sum_{i=1}^k d_{i,n} \, .
\]
The following decomposition holds
\[
X_0 = d_{0,n} + \sum_{k \geq 0} \E ( X'_{k,n} |{\mathcal F}_{-1}) -
\sum_{k \geq 0} \E ( X'_{k+1,n} |{\mathcal F}_0) + X''_{0,n}\, .
\]
Let $h_n = \sum_{k \geq 0} \E ( X'_{k,n} |{\mathcal F}_{-1})$. One
can write
\[
X_i = d_{0,n} \circ \theta^i +h_n\circ \theta^i - h_n\circ
\theta^{i+1} + X''_{0,n} \circ \theta^i
\, ,
\]
and consequently
\[
S_k = M_{k,n} + h_n\circ \theta - h_n\circ \theta^{k+1} + S_{k,n}'' \, ,
\]
with $S_{k,n}'' = \sum_{i=1}^kX''_{0,n} \circ \theta^i$. Hence, for any $x >0$,
\begin{multline} \label{firstdecomposition}
\p ( \max_{1 \leq k \leq n} |S_k| \geq 3x) \leq  \p ( \max_{1 \leq k \leq n} |M_{k,n}| \geq x) \\+
 \p ( \max_{1 \leq k \leq n} |h_n\circ \theta -
  h_n\circ \theta^{k+1}| \geq x)
  + \p ( \max_{1 \leq k \leq n} |S''_{k,n}| \geq x)\, .
\end{multline}

Let us first control the coboundary term. We have
\[
\|\E ( X'_{k,n} |{\mathcal F}_{0}) \|_\infty
\leq \sum_{\ell =1}^L |a_{\ell}| \|\E(g_n \circ
 f_{\ell}(Y_k)|{\mathcal F}_0)-\E(g_n \circ
 f_{\ell}(Y_k))\|_\infty \, .
\]
Applying Lemma \ref{covaphi}, $\|\E(g_n \circ
 f_{\ell}(Y_k)|{\mathcal F}_0)-\E(g_n \circ
 f_{\ell}(Y_k))\|_\infty \leq 4M\phi_{1, {\bf Y}}(k) \sqrt{n}/\sqrt{\log \log n}$. It follows that
\[
  \|h_n\|_\infty \leq 4M \Big(\sum_{k=1}^\infty \phi_{1, {\bf Y}}(k)\Big) \frac{\sqrt n}{\sqrt{\log \log n}} \, .
\]
Hence,  there exists a positive constant $K_1$ such that
\begin{equation} \label{cob}
\sum_{n >0}\frac{1}{n} \p \Big( \max_{1 \leq k \leq n} |h_n\circ \theta - h_n\circ \theta^{k+1}| \geq  C M\sqrt{ n \log \log n}\Big) < K_1 \, .
\end{equation}

Let us now control the large part $X''$. We will prove the existence
of a positive constant $K_2$ such that
\begin{equation}\label{Snseconde}
\sum_{n>0} \frac{1}{n}
\p \Big( \max_{1 \leq k \leq n} |S''_{k,n}| \geq
C M \sqrt{  n \log \log n}\Big) < K_2
\, .
\end{equation}
We shall use the following lemma, whose proof is straightforward:
\begin{lma} \label{lem2}
\[
\p \Big( \max_{1 \leq k \leq n} |S''_{k,n}|  \geq  x\Big)
   \leq \frac{2n}{x}
 \sum_{\ell=1}^L |a_\ell|
 \E(|f_\ell(Y_0)|{ \bf 1}_{|f_\ell(Y_0)| > M n^{1/2}/\sqrt{ \log \log n}} ) \, .
\]
\end{lma}

Applying Lemma~\ref{lem2} with $x=CM \sqrt{n \log \log n}$, we obtain
that
\begin{multline*}
 \p \Big( \max_{1 \leq k \leq n} |S''_{k,n}|
  \geq  C M\sqrt{ n \log \log n}\Big)\\
   \leq \frac{2 n}{ C M \sqrt{  n \log \log n}}
 \sum_{\ell=1}^L |a_\ell|
 \E(|f_\ell(Y_0)|{ \bf 1}_{|f_\ell(Y_0)| > M n^{1/2}/\sqrt{ \log \log n}}).
\end{multline*}
Now, via Fubini, there exists a positive constant $A_1$ such that
\begin{align*}
\sum_{n >0}\frac{1}{n} \frac{n}{\sqrt{n \log \log n}}
\E(| f_\ell(Y_0)|{ \bf 1}_{|f_\ell(Y_0)| >M  n^{1/2}/\sqrt{ \log \log n}} )
< A_1 \|f_\ell(Y_0)\|_2^2  \leq
A_1 M^2\, ,
\end{align*}
and \eqref{Snseconde} follows with $K_2=(2A_1 M)/C$.

Next, we turn to the main term, that is the martingale term. We will prove
that there exists a positive constant $K_3$ such that
\begin{equation}\label{mainpart}
\sum_{n >0}\frac{1}{n} \p \Big( \sup_{1 \leq j \leq n} |M_{j,n}| \geq
C M \sqrt{n \log \log n}\Big) < K_3 \, .
\end{equation}
The main contribution will be controlled through the following
maximal inequality.
\begin{lma}\label{lem3}
Let
\[
c_n=\frac{8M \sqrt n}{\sqrt{\log \log n}} \sum_{k \geq 0}
\phi_{1, {\bf Y}}^{1/2}(k)\, .
\]
The following upper bound holds: for any positive reals $x$ and $y$,
\[
\p \Big ( \sup_{1 \leq j \leq n} |M_{j,n}| \geq x ,
  \sum_{j=1}^n \E ( d_{j,n}^2 | {\mathcal F}_{j-1} ) \leq 2y \Big)
  \leq
  2 \exp \left( -\frac {2y}{c_n^2} \, h
  \Big (  \frac{xc_n}{2y}  \Big )\right)  \, ,
\]
where $h(u)= (1+u) \ln (1+u) -u \geq u \ln (1+u) /2$.
\end{lma}
\begin{proof}
Note first that
\[
\|d_{0,n}\|_\infty \leq 2\sum_{k\geq 0} \|\E(X'_{k,n}|{\mathcal F}_0)\|_\infty
\leq 2 \sum_{k \geq 0} \sum_{\ell =1}^L |a_\ell|
\|\E(g_n \circ
 f_{\ell}(Y_k)|{\mathcal F}_0)-\E(g_n \circ
 f_{\ell}(Y_k))\|_\infty \, .
\]
Now, applying Lemma \ref{covaphi},
\[ \|\E(g_n \circ
 f_{\ell}(Y_k)|{\mathcal F}_0)-\E(g_n \circ
 f_{\ell}(Y_k))\|_\infty \leq \frac{4M \sqrt n}{\sqrt{\log \log n}}
 \phi_{1, {\bf Y}}(k) \, ,
\]
so that
\[
\|d_{0,n}\|_\infty \leq \frac{8M \sqrt n}{\sqrt{\log \log n}} \Big(\sum_{k\geq 0} \phi_{1, {\bf Y}}(k)\Big)\leq c_n \, .
\]
%\[
%\|d_{j,n}\|_{\infty} \leq \frac{2b \sqrt n}{\sqrt{\log \log n}} \sum_{k \geq 0} \phi(k) %\leq \frac{2b \sqrt n}{\sqrt{\log \log n}} \sum_{k \geq 0} \phi^{1/2}(k) : = c_n \, .
%\]
Proposition A.1 in Dedecker, Gou\"ezel and Merlev\`ede (2010) shows that
any sequence of martingale differences $d_j$ which is bounded by a
constant $c$ satisfies
  \[
  \p \Big ( \sup_{1 \leq j \leq n} |M_{j}| \geq x ,
  \sum_{j=1}^n \E ( d_{j}^2 | {\mathcal F}_{j-1} ) \leq 2y \Big)
  \leq
  2 \exp \left( -\frac{2y}{c^2} \, h
  \Big (  \frac{xc}{2y}  \Big )\right)  \, .
\]
The sequence $d_j=d_{j,n}$ satisfies the assumptions of this
proposition for $c=c_n$. Therefore, Lemma \ref{lem3} follows.
\end{proof}

Notice that
 \[\sum_{j=1}^n \E ( d_{j,n}^2)=n\E ( d_{1,n}^2)\leq 4 n
 \Big \Vert \sum_{j \geq 0} \E ( X_{j,n}'|{\mathcal F}_0) \Big \Vert_2^2 \, .
%\leq 4 \times 8^2 \times n  \big ( \sum_{k \geq 0 }
%\phi^{1/2}(k) \big )^2 \|X_0\|_2^2 \, ,
\]
Now,
\[
\|\E ( X'_{k,n} |{\mathcal F}_{0}) \|_2
\leq \sum_{\ell =1}^L |a_{\ell}| \|\E(g_n \circ
 f_{\ell}(Y_k)|{\mathcal F}_0)-\E(g_n \circ
 f_{\ell}(Y_k))\|_2 \, .
\]
Applying Lemma \ref{covaphi}, $\|\E(g_n \circ
 f_{\ell}(Y_k)|{\mathcal F}_0)-\E(g_n \circ
 f_{\ell}(Y_k))\|_2 \leq 2 \sqrt 2\phi_{1, {\bf Y}}^{1/2}(k)\|f_\ell(Y_0)\|_2$.
  It follows that
\[
 \|\E ( X'_{k,n} |{\mathcal F}_{0}) \|_2 \leq 2 \sqrt 2
 \phi_{1, {\bf Y}}^{1/2}(k) M\,
\]
and consequently
\[
\sum_{j=1}^n \E ( d_{j,n}^2)
\leq 32 \,  n  \Big ( \sum_{k \geq 0 }
\phi_{1, {\bf Y}}^{1/2}(k) \Big )^2 M^2\, .
\]
We apply Lemma~\ref{lem3} with
\begin{equation}\label{yn}
y=y_n = 32\,  n  \Big ( \sum_{k \geq 0 } \phi_{1, {\bf Y}}^{1/2}(k) \Big )^2 M^2 \, .
\end{equation}
Letting $x_n = C M \sqrt{ n \log \log n}$, we have
\begin{multline*}
\sum_{n >0}\frac{1}{n} \p \Big( \sup_{1 \leq j \leq n} |M_{j,n}| \geq  x_n ,
  \sum_{j=1}^n \E ( d_{j,n}^2 | {\mathcal F}_{j-1} ) \leq 2y_n \Big)
  \\ \leq 2\sum_{n >0}\frac{1}{n} \exp \Big ( - \frac{x_n}{2c_n} \ln (1 + x_n c_n /(2y_n) )\Big)
 \, .
\end{multline*}
Now, the choice of $C$  imply that
$
 x_n = 4y_n/c_n
$
and $2y_n = c_n^2 (\log \log n)$. It follows that
\[
\sum_{n >0}\frac{1}{n} \exp \Big ( - \frac{x_n}{2c_n} \ln (1 + x_n c_n /(2y_n) )\Big) = \sum_{n >0}\frac{1}{n} \exp \big ( - (\log \log n ) \log 3\big) < \infty \, .
\]

To prove \eqref{mainpart}, it remains to prove that there exists a
positive constant $K_4$ such that
\begin{equation*}
\sum_{n \geq 1 } \frac{1}{n} \p \Big (   \sum_{j=1}^n   \E(d^2_{j,n}|{\mathcal F}_{j-1})   \geq 2y_n \Big )  < K_4 \, .
\end{equation*}
Since  $\sum_{j=1}^n \E ( d_{j,n}^2) \leq y_n$, it suffices to prove
that
\begin{equation} \label{martin}
\sum_{n \geq 1 } \frac{1}{n} \p \Big ( \Big |  \sum_{j=1}^n (  \E(d^2_{j,n}|{\mathcal F}_{j-1}) -\E(d_{j,n}^2))  \Big |  \geq y_n \Big )  <
K_4 \, .
\end{equation}
To prove \eqref{martin}, we shall use the following lemma:

\begin{lma}\label{lem4} If \eqref{DDM} holds,
 there exists a positive constant $C_2(\phi)$ such that for any $y >0$,
\[
\p \Big ( \Big |  \sum_{j=1}^n (  \E(d^2_{j,n}|{\mathcal F}_{j-1}) -\E(d_{j,n}^2))  \Big |  \geq y \Big )
\leq \frac{n C_2(\phi)}{y^2} \sum_{\ell=1}^L |a_{\ell}|
\E(f_\ell (Y_0)^4 {\bf 1}_{|f_\ell (Y_0)| \leq  M n^{1/2}})\, .
\]
\end{lma}
Before proving Lemma~\ref{lem4}, let us complete the proof of
\eqref{martin}, \eqref{mainpart} and \eqref{but2}. Since $y_n$ is
given by \eqref{yn}, we infer from Lemma~\ref{lem4} that there exists
a positive constant $C_3(\phi)$ such that
\begin{multline*}
\sum_{n >0}\frac{1}{n} \p \Big ( \Big |  \sum_{j=1}^n (  \E(d^2_{j,n}|{\mathcal F}_{j-1}) -\E(d_{j,n}^2))  \Big |  \geq y_n \Big ) \\ \leq
\frac{C_3(\phi)}{M^4} \sum_{\ell=1}^L |a_\ell| \sum_{n >0}\frac{1}{n^2}
\E(f_\ell (Y_0)^4 {\bf 1}_{|f_\ell (Y_0)| \leq  M n^{1/2}})\, .
\end{multline*}
By Fubini, the last sum in this equation is bounded by $4 M^2
\|f_\ell(Y_0)\|_2^2  \leq 4 M^4$. Therefore, \eqref{martin} follows
with $K_4= 4 C_3(\phi)$. This completes the proof of
\eqref{mainpart}. Now, the proof of \eqref{but2ter} follows from
\eqref{firstdecomposition}, \eqref{cob}, \eqref{Snseconde} and
\eqref{mainpart}. The inequality \eqref{but2} of
Proposition~\ref{phi} is proved.
\end{proof}

It remains to prove Lemma~\ref{lem4}.

\medskip

\begin{proof}[Proof of Lemma~\ref{lem4}.]
In a sense, the contribution coming from Lemma~\ref{lem4} is less
essential than the contribution we estimated thanks to the maximal
inequality. However, it is rather technical to estimate. To handle
this term, we will argue in the other direction, and go from the
martingale to the partial sums of the original random variables.

We apply  Theorem 3 in Wu and Zhao (2008): for any $q \in (1,2]$
there exists a positive constant $C_q$ such that
\begin{equation*}
\E \Big(\Big |  \sum_{j=1}^n (  \E(d^2_{j,n}|{\mathcal F}_{j-1}) -\E(d_{j,n}^2))  \Big |^q \Big)
\leq C_q n \E (|d_{1,n}|^{2q})  + C_q n \Delta^*_{n,q}
\end{equation*}
where
\[
\Delta^*_{n,q} = \Big ( \sum_{k=1}^{n} \frac{1}{k^{1 + 1 /q}} \|\E(M_{k,n}^2|{\mathcal F}_{0}) - \E(M_{k,n}^2) \|_q \Big )^{q} \, .
\]
Hence, by Markov's inequality with $q=2$, one has
\[
\p \Big ( \Big |  \sum_{j=1}^n (  \E(d^2_{j,n}|{\mathcal F}_{j-1}) -\E(d_{j,n}^2))  \Big |  \geq y \Big )
\leq \frac{C_2n}{y^2}
\Big( \E (|d_{1,n}|^{4})  +  \Delta^*_{n,2}\Big) \, .
\]
Note first that
\[
\E(|d_{1,n}|^{4}) \leq 16  \Big(
 \sum_{j \geq 0} \Vert \E ( X_{j,n}'|{\mathcal F}_0)  \Vert_4 \Big)^4 \, .
 %\leq 16
%\Big(8 \sum_{k>0} \phi_{1, {\bf Y}} (k)^{3/4}\Big)^4
%\sum_{\ell =1}^L |a_{\ell}|
%\E(f_\ell (Y_0)^4 {\bf 1}_{|f_\ell (Y_0)| \leq  M n^{1/2}})\, .
\]
Now
\[
\|\E ( X'_{k,n} |{\mathcal F}_{0}) \|_4
\leq \sum_{\ell =1}^L |a_{\ell}| \|\E(g_n \circ
 f_{\ell}(Y_k)|{\mathcal F}_0)-\E(g_n \circ
 f_{\ell}(Y_k))\|_4 \, .
\]
Applying Lemma \ref{covaphi}, $\|\E(g_n \circ
 f_{\ell}(Y_k)|{\mathcal F}_0)-\E(g_n \circ
 f_{\ell}(Y_k))\|_4 \leq 2(2\phi_{1, {\bf Y}}(k))^{3/4}\|g_n\circ f_\ell(Y_0)\|_4$. It follows that
 \[
\E(|d_{1,n}|^{4})
 \leq
 2^{11} \Big( \sum_{k>0} \phi_{1, {\bf Y}} (k)^{3/4}\Big)^4
\Big(\sum_{\ell =1}^L |a_{\ell}| \|g_n\circ f_\ell(Y_0)\|_4\Big)^4\, .
\]
Applying Jensen's inequality,
\begin{equation}\label{lem1}
\E(|d_{1,n}|^{4})
 \leq 2^{11}
\Big( \sum_{k>0} \phi_{1, {\bf Y}} (k)^{3/4}\Big)^4
\sum_{\ell =1}^L |a_{\ell}|
\E(f_\ell (Y_0)^4 {\bf 1}_{|f_\ell (Y_0)| \leq  M n^{1/2}})\, .
\end{equation}

Now, letting $S_{k,n}' = \sum_{i=1}^k X'_{i,n}$, one has $
M_{k,n}=S_{k,n}' -R_{k,n}$, with
\[
R_{k,n} = \sum_{i \geq 1} \E (X'_{i,n} |{\mathcal F}_0) - \sum_{i \geq k+1} \E (X'_{i,n} |{\mathcal F}_k) \, .
\]
Hence
\begin{multline*}
\Delta^*_{n,2}  \leq 3 \Big ( \sum_{k=1}^{n} \frac{1}{k^{3/2}} \|\E(S_{k,n}'^2 |{\mathcal F}_{0}) - \E(S_{k,n}'^2) \|_2 \Big )^{2} +
3  \Big ( \sum_{k=1}^{n} \frac{1}{k^{3/2}} \|R_{k,n}^2 \|_2 \Big )^{2}\\
+12 \Big ( \sum_{k=1}^{n} \frac{1}{k^{3/2}} \|\E(S_{k,n}'R_{k,n}|{\mathcal F}_{0}) - \E(S_{k,n}'R_{k,n}) \|_2 \Big )^{2}  \, .
\end{multline*}
Arguing as for the proof of  \eqref{lem1}, we obtain that
\begin{align*}
\| R_{k,n}^2\|_2 &
\leq 4 \Big \Vert \sum_{i \geq 1} \E (X'_{i,n} |{\mathcal F}_0) \Big \Vert_4^2
\\&
\leq 32\sqrt 2 \Big( \sum_{k>0} \phi_{1, {\bf Y}} (k)^{3/4}\Big)^2
\Big(\sum_{\ell=1}^L |a_{\ell}| \E(f_\ell (Y_0)^4 {\bf 1}_{|f_\ell
(Y_0)| \leq  M n^{1/2}})\Big)^{1/2}\, .
\end{align*}

From the proof of Corollary 2.1 in Dedecker, Doukhan and Merlev\`ede
(2011), for any $\gamma\in (0,1]$ (to be chosen later), there exists
a positive constant $B$ such that
\begin{equation}\label{borneDDM}
\Big(\sum_{k=1}^{n} \frac{1}{k^{3/2}} \|\E(S_{k,n}'^2 |{\mathcal F}_{0}) - \E(S_{k,n}'^2) \|_2 \Big)^2 \leq B I_1^2 +B I_2^2
\end{equation}
where
\begin{align*}
I_1 &= \sum_{m>0} \frac{m^{\gamma}}{m^{1/2}} \sup_{i \geq j \geq m}
\|\E(X'_{i,n} X'_{j,n}|{\mathcal F}_0)-\E(X_{i,n}'X_{j,n}')\|_2 \\
I_2 &= \Big(\sum_{k>0} \frac{k^{1/(2\gamma)}}{k^{1/4}}
\|\E(X_{k,n}'|{\mathcal F}_0)\|_4\Big)^2\, .
\end{align*}
Arguing as for the proof of \eqref{lem1}, we obtain that
\begin{equation} \label{I2}
I_2 \leq 8 \sqrt 2 \Big(  \sum_{k>0} \frac{k^{1/(2\gamma)}}{k^{1/4}}
\phi_{1, {\bf Y}}(k)^{3/4} \Big)^2
\Big(\sum_{\ell =1}^L |a_{\ell}|
\E(f_\ell (Y_0)^4 {\bf 1}_{|f_\ell (Y_0)| \leq  M n^{1/2}})\Big)^{1/2} \, .
\end{equation}
To bound $I_1$, note that
\begin{multline*}
 \|\E(X'_{i,n} X'_{j,n}|{\mathcal F}_0)-\E(X_{i,n}'X_{j,n}')\|_2
 \\
 \leq \sum_{k=1}^L \sum_{\ell=1}^L |a_k||a_\ell| \|
 \E((g_n\circ f_k (Y_i))^{(0)} (g_n \circ f_\ell (Y_j))^{(0)}|{\mathcal F}_0)-
 \E((g_n\circ f_k (Y_i))^{(0)} (g_n \circ f_\ell (Y_j))^{(0)})\|_2\, .
\end{multline*}
Applying Lemma \ref{covaphi}, for $i \geq j \geq m$,
\begin{multline*}
\|\E((g_n\circ f_k (Y_i))^{(0)} (g_n \circ f_\ell (Y_j))^{(0)}|{\mathcal F}_0)-
 \E((g_n\circ f_k (Y_i))^{(0)} (g_n \circ f_\ell (Y_j))^{(0)})\|_2 \\
 \leq 16  \phi_{2, {\bf Y}}(m)^{1/2} \|g_n \circ f_k(Y_0)\|_4
 \|g_n \circ f_\ell (Y_0)\|_4 \, .
\end{multline*}
It follows that
\begin{equation}\label{I1}
 I_1 \leq \Big(16 \sum_{m>0} \frac{m^{\gamma}}{m^{1/2}} \phi_{2, {\bf Y}}(m)^{1/2}\Big)
 \Big(\sum_{\ell =1}^L |a_{\ell}|
\E(f_\ell (Y_0)^4 {\bf 1}_{|f_\ell (Y_0)| \leq  M n^{1/2}})\Big)^{1/2}\, .
\end{equation}
Let $\gamma=1/\sqrt 3$. If the condition \eqref{DDM} holds, then
\[
\sum_{k>0} \frac{k^{{\sqrt 3}/2}}{k^{1/4}}
\phi_{1, {\bf Y}}(k)^{3/4}< \infty \quad \text{and} \quad
\sum_{m>0} \frac{m^{1/\sqrt 3}}{m^{1/2}} \phi_{2, {\bf Y}}(m)^{1/2} < \infty.
\]
To see that the convergence of the second series implies the
convergence of the first series, it suffices to note that $\phi_{1,
{\bf Y}}(n)\leq \phi_{2, {\bf Y}}(n)$ and that, since $\phi_{2, {\bf
Y}}(n)$ is nonincreasing, $\phi_{2, {\bf Y}}(n)=o(n^{-(2+\sqrt
3)/\sqrt 3})$.

We infer  from \eqref{borneDDM}, \eqref{I2} and \eqref{I1} that, if
\eqref{DDM} holds, there exists a positive constant $C_4(\phi)$ such
that
\begin{equation}\label{borne}
\Big(\sum_{k=1}^{n} \frac{1}{k^{3/2}} \|\E(S_{k,n}'^2 |{\mathcal F}_{0}) - \E(S_{k,n}'^2) \|_2 \Big)^2 \leq C_4(\phi)
\sum_{\ell =1}^L |a_{\ell}|
\E(f_\ell (Y_0)^4 {\bf 1}_{|f_\ell (Y_0)| \leq  M n^{1/2}})\, .
\end{equation}

Let us consider now the term
\begin{equation*}
 \Big ( \sum_{k=1}^{n} \frac{1}{k^{3/2}} \|\E(S_{k,n}'R_{k,n}|{\mathcal F}_{0}) - \E(S_{k,n}'R_{k,n}) \|_2 \Big )^{2}  \, .
\end{equation*}
As for the proof of \eqref{lem1}, one has
\begin{align*}
\Big \|\E(S_{k,n}'|{\mathcal F}_{0})
\sum_{i \geq 1} \E (X'_{i,n} |{\mathcal F}_0) \Big \|_2^2
& \leq  \Big ( \sum_{i \geq 1}
\Vert \E (X'_{i,n} |{\mathcal F}_0) \|_4 \Big )^2 \\
& \leq  8 \sqrt 2 \Big ( \sum_{i \geq 1} \phi_{1, {\bf Y}}^{3/4}(i) \Big )^2
\Big(\sum_{\ell =1}^L |a_{\ell}|
\E(f_\ell (Y_0)^4 {\bf 1}_{|f_\ell (Y_0)| \leq  M n^{1/2}})\Big)^{1/2}\, .
\end{align*}
Next, we need to bound
\begin{equation*}
  \Big ( \sum_{k=1}^{n} \frac{1}{k^{3/2}} \Big \|\E(S_{k,n}'\sum_{i \geq k+1}
  \E (X'_{i,n} |{\mathcal F}_k)|{\mathcal F}_{0}) - \E(S_{k,n}'\sum_{i \geq k+1} \E (X'_{i,n} |{\mathcal F}_k)) \Big \|_2 \Big )^{2}  \, .
\end{equation*}
First, we see that
\begin{equation*}
\sum_{i \geq k+1} \E (X'_{i,n} |{\mathcal F}_k) =  \E (S_{2k,n}'- S_{k,n}' |{\mathcal F}_k)  + \sum_{j \geq 2k+1} \E (X'_{j,n} |{\mathcal F}_k)   \, .
\end{equation*}
Since $S_{k,n}'$ is ${\mathcal F}_k$-measurable, we get that
\begin{multline*}
\Vert \E ( S'_{k,n} \E (S_{2k,n}'- S_{k,n}'
|{\mathcal F}_k)|{\mathcal F_0}) - \E ( S'_{k,n} \E (S_{2k,n}'- S_{k,n}'  |{\mathcal F}_k)) \Vert_{2} \\
= \Vert  \E (S'_{k,n}  (S_{2k,n}'- S_{k,n}' ) |{\mathcal F_0}) - \E  ( S'_{k,n} (S_{2k,n}'- S_{k,n}'  ) )\Vert_{2} \, .
\end{multline*}
Next using the identity $2ab = (a+b)^2 - a^2 -b^2$ and the stationarity, we obtain that
\begin{multline*}
2  \Vert \E ( S'_{k,n} \E (S'_{2k,n}- S_{k,n}' |{\mathcal F}_k)|{\mathcal F_0}) - \E ( S'_{k,n} \E (S'_{2k,n}- S_{k,n}'  |{\mathcal F}_k)) \Vert_{2}  \\
 \quad \leq \Vert  \E (S_{2k,n}'^2|{\mathcal F_0}) - \E ( S_{2k,n}'^2 ) \Vert_{2} + 2 \Vert  \E (S_{k,n}'^2|{\mathcal F_0}) -
\E ( S_{k,n}'^2 ) \Vert_{2} \, ,
\end{multline*}
which combined with \eqref{borne} implies that
\begin{multline*}   \Big ( \sum_{k=1}^{n} \frac{1}{k^{3/2}}\Vert \E ( S'_{k,n} \E (S'_{2k,n}- S_{k,n}' |{\mathcal F}_k)|{\mathcal F_0}) - \E ( S'_{k,n} \E (S'_{2k,n}- S_{k,n}'  |{\mathcal F}_k)) \Vert_{2} \Big )^{2} \\ \leq
6 \, C_4(\phi)
\sum_{\ell =1}^L |a_{\ell}|
\E(f_\ell (Y_0)^4 {\bf 1}_{|f_\ell (Y_0)| \leq  b n^{1/2}}) \, .
\end{multline*}
It remains to bound
\begin{equation*} \label{but7}
 \Big ( \sum_{k=1}^{n} \frac{1}{k^{3/2}}
 \Big \Vert \E \Big ( S'_{k,n} \sum_{j \geq 2k+1} \E (X'_{j,n} |{\mathcal F}_k)  \Big |{\mathcal F_0}\Big ) \Big \Vert_{2} \Big )^{2}  \, .
\end{equation*}
By stationarity,
\begin{equation*} \label{bound}
\sum_{j \geq 2k+1} \Vert \E (S'_{k,n} \E ( X'_{j,n}|{\mathcal F}_k)
 ) | {\mathcal F}_0) \Vert_{2}
 \leq   k \sum_{j \geq k+1} \Vert X'_{0,n} \E ( X'_{j,n}|{\mathcal F}_0 ) \Vert_{2} \, .
\end{equation*}
Now, as for the proof of \eqref{lem1},
\begin{align*}
\sum_{j \geq k+1}\Vert X'_{0,n} \E ( X'_{j,n}|{\mathcal F}_0 ) \Vert_{2} &\leq
\Vert  X'_{0,n}\Vert_4\sum_{j \geq k+1} \Vert \E ( X'_{j,n} |{\mathcal F}_0) \Vert_{4}
\\
&\leq 2 \Big(2\sum_{j \geq k+1} (2\phi_{1, {\bf Y}}(j))^{3/4}\Big) \Big(\sum_{\ell =1}^L |a_{\ell}|
\E(f_\ell (Y_0)^4 {\bf 1}_{|f_\ell (Y_0)| \leq  b n^{1/2}})\Big)^{1/2}\, ,
\end{align*}
and consequently, there exists a positive constant $D$ such that
\begin{multline*}
 \Big ( \sum_{k=1}^{n} \frac{1}{k^{3/2}}
 \Big \Vert \E \Big ( S'_{k,n} \sum_{j \geq 2k+1} \E (X'_{j,n} |{\mathcal F}_k)  \Big |{\mathcal F_0}\Big) \Big \Vert_{2} \Big )^{2} \\
 \leq
 D \Big(\sum_{j \geq 2} j^{1/2} \phi_{1, {\bf Y}}(j)^{3/4}\Big)^2
 \sum_{\ell =1}^L |a_{\ell}| \E(f_\ell (Y_0)^4 {\bf 1}_{|f_\ell (Y_0)| \leq  M n^{1/2}})\, .
\end{multline*}
The lemma is proved.
\end{proof}

\section{Proof of Theorem~\ref{ASmap1} on uniformly expanding maps}
\label{secproofasmap1}

Let $(Y_i)_{i \geq 0}$ be the stationary Markov chain with transition
kernel $K$ corresponding to the iteration of the inverse branches of
$T$, and let $X_n=f(Y_n)-\nu (f)$. Concerning Item 1 in
Theorem~\ref{ASmap1}, it is well known that it is equivalent to prove
it for the iteration of the map or of the Markov chain, since the
distributions are the same (see for instance the proof of Theorem 2.1
in Dedecker and Merlev\`ede (2009)). Therefore, it is enough to show
that the process
\[
   \Big \{\frac{1}{\sqrt n} \sum_{i=1}^{[nt]} X_i
  , \ t \in [0,1]\Big \}
\]
converges in distribution in the Skorokhod topology to $\sigma W$,
where $W$ is a standard Wiener process. Now, as shown by Heyde
(1975), this property as well as the absolute convergence of the
series \eqref{var} will be true provided that Gordin's condition
(1969) holds, that is
\begin{equation}\label{Gor}
  \sum_{n=0}^\infty \| K^n(f)-\nu(f)\|_{\mathbb{L}^2(\nu)} < \infty \, .
\end{equation}

By definition of $\Monm_2(M,\nu)$, there exists a sequence of
functions $f_L= \sum_{\ell=1}^L a_{\ell, L} g_{\ell, L}$ with
$g_{\ell, L}$ belonging to $\Mon_2(M, \nu)$ and $\sum_{\ell=1}^L
|a_{\ell, L}| \leq 1$, such that $f_L$ converges in ${\mathbb
L}^1(\nu)$ to $f$. It follows that, for any nonnegative integer $n$,
$K^n(f_L)-\nu(f_L)$ converges to $K^n(f)-\nu(f)$ in ${\mathbb
L}^1(\nu)$. Hence, there exists a subsequence
$K^n(f_{\varphi(L)})-\nu(f_{\varphi(L)})$ converging to $K^n(f)-\nu(f)$
almost surely and in ${\mathbb L}^1(\nu)$. Applying Fatou's lemma, we
infer that
\begin{equation}\label{Fat}
\| K^n(f)-\nu(f)\|_{\mathbb{L}^2(\nu)} \leq \liminf_{L \rightarrow \infty}
\| K^n(f_{\varphi(L)})-\nu(f_{\varphi(L)})\|_{\mathbb{L}^2(\nu)} \, .
\end{equation}
Applying Lemma \ref{covaphi}, for any $g$ in $\Mon_2(M, \nu)$, $\|
K^n(g)-\nu(g)\|_{\mathbb{L}^2(\nu)} \leq 2 \sqrt 2 \phi_{1, {\bf
Y}}^{1/2}(n) M$. Hence
\[
\| K^n(f_{\varphi(L)})-\nu(f_{\varphi(L)})\|_{\mathbb{L}^2(\nu)} \leq
\sum_{\ell=1}^L |a_{\ell, L}|
\| K^n(g_{\ell, \varphi(L)})-\nu(g_{\ell, \varphi(L)})\|_{\mathbb{L}^2(\nu)}
\leq 2 \sqrt 2 \phi_{1, {\bf
Y}}^{1/2}(n) M \, .
\]
From \eqref{Fat}, it follows that $ \|
K^n(f)-\nu(f)\|_{\mathbb{L}^2(\nu)} \leq 2 \sqrt 2 \phi_{1, {\bf
Y}}^{1/2}(n) M$, and \eqref{Gor} holds provided that $\sum_{n>0}
\phi_{1, {\bf Y}}^{1/2}(n) < \infty$. Now, if $T$ is uniformly
expanding, it follows from Section 6.3 in Dedecker and Prieur (2007)
that
 $\phi_{2, {\bf Y}}(n)=O(\rho^n)$ for some $\rho
\in (0, 1)$, and Item 1 is proved.

According to the inequality (4.1) in Dedecker, Gou\"ezel and Merlev\`ede
(2010), we have
\[
\nu \Big( \max_{1 \leq k \leq
  n} \Big |\sum_{i=0}^{k-1} (f \circ T^i
  -\nu(f))\Big|> x\Big)
  \leq
 \nu \Big( 2 \max_{1 \leq k \leq
  n} \Big |\sum_{i=1}^{k} X_i \Big|> x \Big) \, .
\]
Therefore, Item 2 follows from Proposition~\ref{phi} applied to the
sequences $(X_i)_{i \geq 1}$ as soon as \eqref{DDM} holds,  which is
clearly true.

For Item 3, we proceed exactly as in the case of GPM maps, relying on
the approximation $f=\bar f_m+g_m$ given by Lemma~\ref{lem_approxL2}
to apply Proposition~\ref{diag}. Since $g_m \in \Monm_2(1/m, \nu)$,
Proposition~\ref{phi} shows that almost surely
  \begin{equation*}
  \limsup \frac{1}{\sqrt{n\log \log n}} \left|\sum_{i=0}^{n-1} (g_m \circ T^i-\nu(g_m))\right| \leq C/m\, ,
  \end{equation*}
for some constant $C$. Moreover, the proof of Theorem 3.1 in
Merlev\`ede and Rio (2012) shows that the sequence $\bar f_m \circ
T^i-\nu(f_m)$ satisfies an almost sure principle, towards a Gaussian
with variance $\sigma_m^2$. It only remains to show that $\sigma_m^2$
converges to $\sigma^2$. We start from the basic inequality
 \[
 \sigma^2(g_m) \leq 2 \|g_m\|_{\mathbb{L}^2(\nu)}\sum_{n=0}^\infty
 \|K^n(g_m) - \nu(g_m)\|_{\mathbb{L}^2(\nu)}\, .
 \]
Arguing as in \eqref{Fat}, we infer that
\[
\sigma^2(g_m) \leq 16 m^{-2} \sum_{k=0}^\infty\phi_{1, {\bf Y}}^{1/2}(k)
\]
and the series on the right hand side is finite since $\phi_{1, {\bf
Y}}(n)=O(\rho^n)$ for some $\rho \in (0, 1)$. Therefore,
$\sigma^2(g_m)$ converges to $0$. \qed

%%%%%%%%%%%%%%%%%%%%%%%%%%%%%%% BIBLIOGRAPHY %%%%%%%%%%%%%%%%%%%%%%

\end{document}